\documentclass[11pt]{amsart}
\usepackage{graphics}
\usepackage{hyperref}

\newtheorem{theorem}{Theorem}[section]
\newtheorem{lem}[theorem]{Lemma}
\newtheorem{prop}[theorem]{Proposition}

\theoremstyle{definition}
\newtheorem{definition}[theorem]{Definition}
\newtheorem{example}[theorem]{Example}

\theoremstyle{remark}

\numberwithin{equation}{section}

\newcommand{\N}{\mathbb{N}}
\newcommand{\Z}{\mathbb{Z}}
\newcommand{\ra}{\rightarrow}
\newcommand{\U}{\mathcal{U}}
\newcommand{\bm}{\mathfrak{b}}
\newcommand{\hB}{\hat{\partial}}

\newcommand{\C}{\mathcal{C}}
\newcommand{\R}{\mathbb{R}}

\begin{document}

\title{Spanning Tree bounds for grid graphs}

\author {Kristopher Tapp}
\email{ktapp@sju.edu}
\begin{abstract}
Among subgraphs with a fixed number of vertices of the regular square lattice, we prove inequalities that essentially say that those with smaller boundaries have larger numbers of spanning trees and vice-versa.  As an application, we relate two commonly used measurements of the compactness of district maps.
\end{abstract}
\maketitle

\section{Introduction}
For a finite connected graph $G$, let $\tau(G)$ denote its number of spanning trees.  The study of this measurement goes back to Kirchoff's Matrix-Tree Theorem, which equates it with the product of the non-zero eigenvalues of the Laplacian of $G$~\cite{Kir}.

Let $\mathcal{L}(\Z^2)$ denote the regular square lattice, which has vertex set $\Z^2$ and rook-adjacent edges.  We are interested here in \emph{grid graphs}, by which we mean finite connected subgraphs of $\mathcal{L}(\Z^2)$.  The \emph{bulk limit} of $\mathcal{L}(\Z^2)$ is known to equal $\frac{4C}{\pi}$, where $C$ is Catalan's constant.  This means that
\begin{equation}\label{E:Bulklim}\lim_{k\ra\infty} \frac{\ln(\tau(G(k)))}{|V(G(k))|} = \frac{4C}{\pi}\approx 1.166243,
\end{equation}
where $G(1)\subset G(2)\subset\cdots$ is any nested sequence of grid graphs (satisfying certain weak hypotheses) whose union equals $\mathcal{L}(\Z^2)$.  In this paper, $V(G)$ and $E(G)$ will denote the vertex set and edge set of a graph $G$.  For this theorem and analogous results for other lattices, see~ \cite{BP}, \cite{Shrock}, \cite{Temp}, \cite{Teufl}, \cite{Wu} and references therein.

There is a good intuition that, among grid graphs with a fixed number of vertices (or even among more general classes of graphs), the ones with higher numbers of spanning trees should have smaller boundaries and vice-versa.  Asymptotic evidence for this intuition is found in~\cite{Kenyon}.

Further evidence comes from recent work on the mathematics of redistricting.  For the reversible version of the spanning-tree-based MCMC algorithm by which ensembles of maps are commonly generated, the stationary distribution is known to assign a probability to each map that is proportional to the map's \emph{spanning tree score} (which means the product of the numbers of spanning trees of its districts); see~\cite{revRecom}, \cite{Mergesplit}, \cite{Recom}.  Thus, the algorithm prefers maps whose districts have larger numbers of spanning trees.  Data from large ensembles of maps indicate a strong negative correlation between a map's spanning tree score and its number of cut edges (which is a discrete measurement of the total size of the district boundaries)~\cite{revRecom}.  Thus, the algorithm seems to prefer maps whose districts have small boundaries, and one purpose of this paper is to more rigorously understand this behavior.

The following important result, due to Russell Lyons, says that the bulk limit is an upper bound:
\begin{theorem}[Lyons]\label{T:Lyons} If $G$ is a grid graph, then
$$\ln(\tau(G)) < \frac{4C}{\pi}\cdot |V(G)|.$$
\end{theorem}
For completeness, we'll include Lyons' unpublished proof of this theorem in the next section.  In terms of the base
\begin{equation*}\mathfrak{b}=\exp(4C/\pi)\approx 3.2099,
\end{equation*}
Lyons' theorem can be re-phrased as:
\begin{equation}\label{E:Lyons_rephrase}
\tau(G)<\bm^{|V(G)|}.
\end{equation}

Our main result is related to this, and is easiest to state for the following natural class of grid graphs.

\begin{definition}\label{D:simple} A \emph{simple closed loop $\alpha$ in} $\mathcal{L}(\Z^2)$ is a sequence of more than two distinct vertices where each successive pair (including the pair consisting of the first and the last) is an edge of $\mathcal{L}(\Z^2)$.  A grid graph $G$ is called \emph{simple} if it is comprised of all of the vertices and edges that are on and interior to a simple closed loop $\alpha$ in $\mathcal{L}(\Z^2)$.  In this case, the set vertices of $\alpha$ is called the \emph{boundary} of $G$, denoted $\partial G$.  The \emph{area of $G$}, denoted $\text{Area}(G)$, means the area of the interior of $\alpha$, or equivalently the number of faces of $G$.
\end{definition}

\begin{figure}[ht!]\centering
   \scalebox{1.5}{\includegraphics{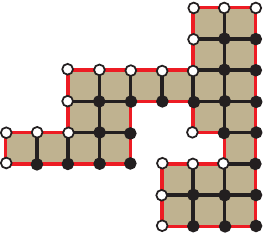}}
\caption{A simple grid graph.  The white vertices lie in its \emph{top-left boundary.}}\label{F:GridGraph}
   \end{figure}

Figure~\ref{F:GridGraph} illustrates a simple grid graph $G$.  Its \emph{bounding loop} $\alpha$, colored red, can be considered as a piecewise-linear path in $\R^2$ whose length equals $|\partial G|$.  The white vertices comprise the \emph{top-left boundary} of $G$, defined as:
\begin{definition}\label{D:simple_Special} The \emph{top-left boundary} of a simple grid graph $G$, denoted $\hB G$, is the set of all $v\in\partial G$ such that the face of $\mathcal{L}(\Z^2)$ whose bottom-right corner is $v$ is not a face of $G$.
\end{definition}

Our main result is the following theorem, which forces graphs with larger boundaries have smaller numbers of spanning trees and vice-versa:

\begin{theorem}\label{T:main} If $G$ is a simple grid graph, then
$$\bm^m\leq \tau(G)\leq\ 4^m,$$
where $m = \text{Area}(G) = |V(G)|-\frac 12 |\partial G| - 1 = |V(G)|-|\hB G|$.
\end{theorem}
The fact that $\text{Area}(G) = |V(G)|-\frac 12 |\partial G| - 1$ follows from Pick's Theorem~\cite{Pick} or from an elementary argument that we'll include in Section 4, where we'll also prove that $\hB G$ contains one more than half of the vertices of $\partial G$.

We'll show that the lower bound of Theorem~\ref{T:main} more generally makes sense and is true for \emph{all} grid graphs, but that the upper bound is only true of \emph{simple} grid graphs.

The example of a $1$-by-$1$ square demonstrates that the upper bound of Theorem~\ref{T:main} is sharp.  However when $\frac{|\partial G|}{|V(G)|}$ becomes small, this upper bound becomes worse than Theorem~\ref{T:Lyons}.  To improve this situation, we give a much stronger upper bound in Section 6.

This paper is organized as follows.  Section 2 contains Lyons' unpublished proof of Theorem~\ref{T:Lyons}.  Section 3 explains the main idea of this paper with an illuminating example.  Section 4 derives basic properties of the top-left boundary of a grid graph.  Sections 5 and 6 respectively prove the lower and upper bound of Theorem~\ref{T:main} plus generalizations and improvements.

Finally in Section 7 we apply our main theorem to relate two different measurements of compactness that are commonly used in the mathematical redistricting literature: a map's cut edge count and its spanning tree score.  Empirical evidence suggests a very strong negative correlation between these two measurements, and our results partially account for this correlation.  Independent work by Procaccia and Tucker-Folz related these two measurements for general planar graphs~\cite{JamieAriel}; in the case of grid graphs, our results are complimentary to theirs.

\section*{Acknowledgments} The author is pleased to thank Russell Lyons for valuable feedback.

\section{The bulk limit is an upper bound}
We thank Russell Lyons for sharing with us the following proof.  For brevity, in this section we assume knowlege of the vocabulary and results of~\cite{Lyons_pre} and~\cite{Lyons}.

\begin{proof}[Proof of Theorem~\ref{T:Lyons}]
Choose a leftmost vertex $x$ of $G$ and a rightmost vertex $y$ of $G$.  For every integer $k\in\Z$, let $G_k$ be a copy of $G$ with corresponding vertices named $x_k$ and $y_k$.  For every integer $n>0$, let $H_n$ denote the connected graph formed from all of the copies $G_k$ with $-n \le k \le n$, with the copies connected together by adding an edge between $y_k$ and $x_{k+1}$ for each $-n \le k < n$.  Notice that each $H_n$ is isomorphic to a grid graph; that is, the construction can be embedded in  $\mathcal{L}(\Z^2)$.

We have $\tau(H_n) = \tau(G)^{2n+1}$ and $|V(H_n)| = (2n+1)|V(G)|$.  Therefore,
$$\frac{\ln (\tau(G))}{|V(G)|} = \frac{\ln (\tau(H_n))}{|V(H_n)|}.$$
By~\cite[Theorem 3.2]{Lyons_pre}, the limit of the latter quantity is the tree entropy of the random rooted infinite graph $H_\infty$ formed similarly from \emph{all} copies $G_k$ and rooted at a uniformly random vertex of $G_0$. Clearly $H_\infty$ is stochastically dominated by the entire square lattice $\mathcal{L}(\Z^2)$, whence the tree entropy of $H_\infty$ is strictly less than that of $\mathcal{L}(\Z^2)$ by~\cite[Theorem 3.2]{Lyons}.  The latter is $4C/\pi$, which proves the claimed upper bound.
\end{proof}

Notice that this proof generalizes to yield the analogous result for any lattice (in any dimension) with a transitive group of translation symmetries.

\section{Setup and example}
In this section, let $G$ be a grid graph.  Our main technique involves building $G$ by adding one vertex at a time in the words-on-a-page order (starting with the top row ordered left-to-right and ending with the bottom row ordered left-to-right), and studying the multiplicative factor by which the spanning tree count grows with each added vertex.

More precisely, let $v\in V(G)$. Let $H'_v$ (respectively $H_v$) denote the subgraph of $G$ induced by all vertices prior to $v$ (respectively prior to and including $v$) with respect to the words-on-a-page ordering of $V(G)$.  We will study the multiplicative growth factor:
$$m_v = \frac{\tau(H_v)}{\tau(H'_v)}.$$
To allow for the possibility of disconnected graphs, the meaning of $\tau$ here must be slightly generalized as follows.  If $H$ is a (possibly disconnected) graph, let $\mathcal{T}(H)$ denote the set of ways to choose one spanning tree from each of its connected components, and let $\tau(H) = |\mathcal{T}(H)|$, which equals the product of the numbers of spanning trees on the connected components.  Our convention here is that a component containing just a single vertex is counted as having one spanning tree.  We additionally use the convention that $m_v=1$ if $v$ is the first vertex.  With these definitions, we can recover $\tau(G)$ as:
$$\tau(G) = \prod_{v\in V(G)} m_v.$$

It is useful to regard $v\mapsto m_v$ as real-valued function on $V(G)$, which we call the \emph{multiplier function}.  In fact, the primary technical goal of this paper is to understand its behavior on general grid graphs.  For this, it is helpful to first gain intuition from examples.

\begin{example} Figure~\ref{F:heatmaps} illustrates heatmaps for the multiplier function on two grid graphs.  The left grid graph, which we call $S$, is the $12$-by-$12$ square.  The right grid graph, which we call $D$, is the diamond inside the $17$-by-$17$ square.  These examples were chosen to have similar numbers of vertices: $|V(S)|=144$, while $|V(D)| = 145$.  Each small square represents a vertex.  The graphs' edges don't need to be displayed because adjacency is visually obvious, so the small squares are drawn large enough to bump into their neighbors forming a grid.  The color of each small square represents the value of the multiplier function on the corresponding vertex.

For each vertex $v$ of $S$ or $D$, the underlying data shows that either $m_v=1$ or $m_v\in(\mathfrak{b},4]$.  The set of vertices with multiplier $1$ (colored black) is exactly the \emph{top-left boundary}.  The square's top-left boundary has $23$ vertices, while the diamond's has $33$.  The square has more spanning trees: $\ln(\tau(S))\approx 146.15, \,\,\,\,\,\ln(\tau(D))\approx 136.19$.  The diamond is not simple, but it becomes simple if its four degree-one vertices (the top-most, bottom-most, right-most and left-most vertices) are removed; this removal doesn't affect the spanning tree count.  After this removal, the square and the diamond are both simple grid graphs with the property that the size of the top-left boundary is one more than half the size of the boundary.
\end{example}
\begin{figure}[ht!]\centering
   \scalebox{.65}{\includegraphics{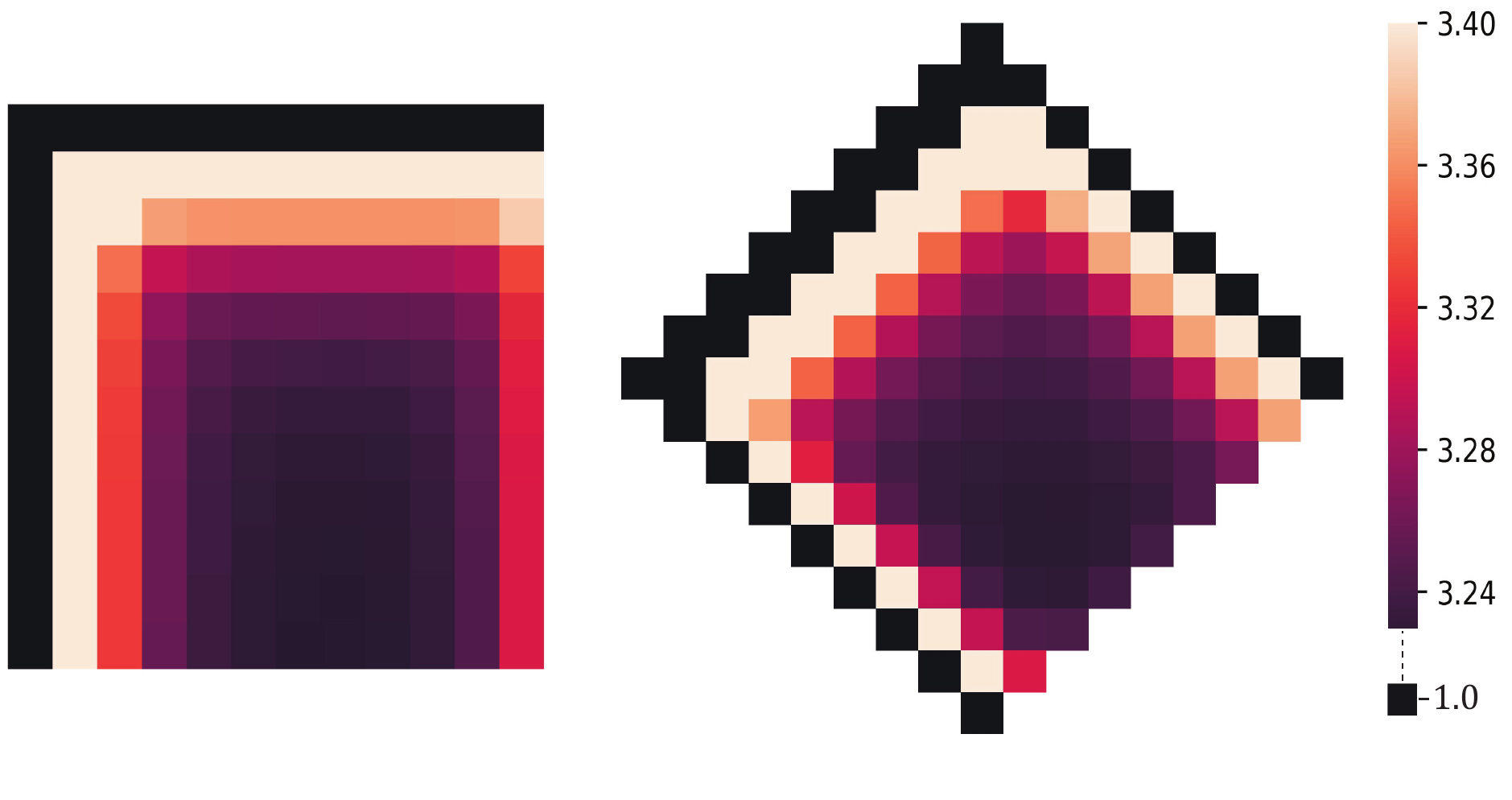}}
\caption{Heatmaps for the multiplier function on a square and a diamond}\label{F:heatmaps}
   \end{figure}

The remainder of this paper will demonstrate that each key feature of the previous examples carries over to all grid graphs or at least all simple grid graphs.

\section{The top-left boundary}
In this section, we study the \emph{top-left boundary} and prove that it behaves like the set of black-colored vertices in the examples of the previous section.  We begin by generalizing Definition~\ref{D:simple_Special} to (not necessarily simple) grid graphs.

\begin{definition}\label{D:topleft} Let $G$ be a grid graph. For each $v\in V(G)$, let ${}^\Box v$ denote the subgraph of $\mathcal{L}(\Z^2)$ comprised of the vertices and edges of the $1$-by-$1$ square whose bottom-right corner is $v$.  The \emph{top-left boundary} of $G$ is:
$$\hat{\partial} G = \{v\in V(G)\mid {}^\Box v\text{ is not a subgraph of }G\}.$$
\end{definition}

\begin{lem}\label{L:toplefty} If $G$ is simple and $v\in\hB G$, then $m_v=1$.
\end{lem}

\begin{proof}
Denote the coordinates of $v$ as $v=(x,y)$.  Denote the relevant neighbors of $v$ as $a=(x,y+1)$, $w=(x-1,y+1)$ and $b=(x-1,y)$.  The following three cases are straightforward:
\begin{itemize}
\item If $\overline{va}\notin E(G)$ and $\overline{vb}\notin E(G)$, then $\{v\}$ is a connected component of $H_v$, so $m_v=1$.
\item If $\overline{va}\in E(G)$ and $\overline{vb}\notin E(G)$, then $|\mathcal{T}(H'_v)|=|\mathcal{T}(H_v)|$ because ``adding the edge $\overline{va}$'' is a bijection between $\mathcal{T}(H'_v)$ and $\mathcal{T}(H_v)$, so $m_v=1$.
\item If $\overline{va}\notin E(G)$ and $\overline{vb}\in E(G)$, then $m_v=1$ by a similar argument.
\end{itemize}
Next assume that  $\overline{va}\in E(G)$ and  $\overline{vb}\in E(G)$, which is the only remaining case.  We claim that $a$ and $b$ must lie in different connected components of $H'_v$.  Indeed, if there were a path between $a$ and $b$ in $H'_v$, then adding $\overline{va}$ and $\overline{vb}$ to this path would yield a loop in $H_v$ that encloses or contains $w$.  But since $G$ is simple, it contains all edges inside of any closed loop in it, so $\overline{aw},\overline{bw}\in E(G)$, contradicting the hypothesis that $v\in\hB G$.

In summary, $a$ and $b$ lie in different connected components of $H_v'$, but they are connected through $v$ in $H_v$.  Therefore, ``adding $\overline{va}$ and $\overline{vb}$'' is a bijection between $\mathcal{T}(H_v')$ and $\mathcal{T}(H_v)$, so $m_v=1$.
\end{proof}

Figure~\ref{F:counterexample} exhibits counterexamples to Lemma~\ref{L:toplefty} when $G$ is not simple.  Each graph has the property that all of its vertices lie in its top-left boundary, but yet its red-colored vertices have multipliers larger than $1$.  In fact, the red vertex of the left graph has multiplier $16$; we'll soon see that this is much larger than the multiplier of \emph{any} vertex of a simple graph.

\begin{figure}[ht!]\centering
   \scalebox{.5}{\includegraphics{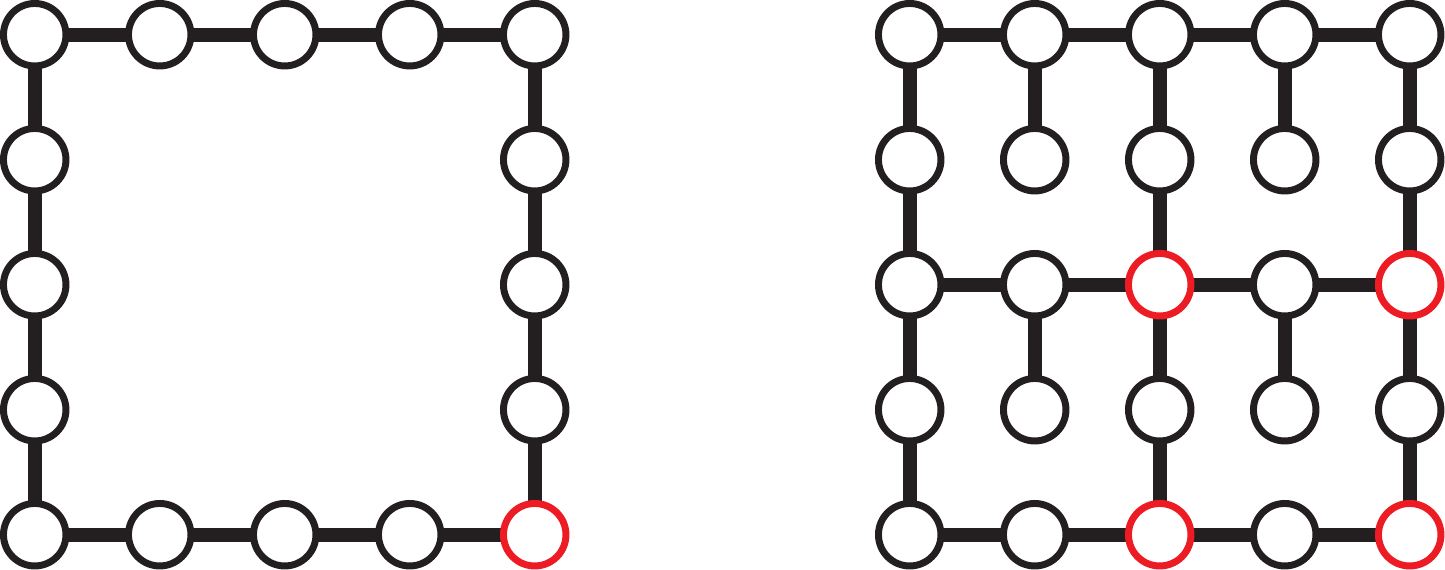}}
\caption{Counterexamples to Lemma~\ref{L:toplefty} when $G$ is not simple.}\label{F:counterexample}
   \end{figure}

The decision to consider the top-left boundary (rather than the top-right, bottom-left, or bottom-right) is somewhat arbitrary, but the size of $\hB G$ is unaffected by this decision because the following proposition provides a canonical interpretation of $|\hB G|$.
\begin{lem}\label{L:count} If $G$ is simple, then $\hB G\subset \partial G$, and $|\partial G|$ is even, and $$|\hB G| = \frac12|\partial G| +1.$$
\end{lem}
\begin{proof} The claim that $\hB G\subset\partial G$ is straightforward.  The proof of the other two assertions is by induction.  The assertions are clearly true for a $1$-by-$1$ square, which is the only simple grid graph with $4$ or fewer vertices.  Now suppose that the assertions are true of all simple grid graphs with $\leq N$ vertices.  Let $G$ be a simple grid graph with $N+1$ vertices.

Consider a line of slope $1$ in $\mathbb{R}^2\supset\Z^2$ positioned far left, and shift it right until it first touches a vertex $v$ of $G$.  This line is illustrated as a dashed black line in Figure~\ref{F:half_boundary}.  Let $f$ be the face of $G$ whose top-left corner is $v$.  Let $G'$ be the grid graph obtained by removing $v$ and also removing each other edge and vertex that belongs only to $f$ (not to any other face of $G$).

The cases that must be considered are exemplified in Figure~\ref{F:half_boundary}.  In case 1, the right neighbor of $f$ is a face $G$ but the bottom neighbor is not.  Here $G'$ is a simple grid graph, and the inductive hypothesis applied to $G'$ leads to the desired result for $G$.  The case for which the bottom neighbor of $f$ is a face $G$ but the right neighbor is not is handled similarly.

In case $2$, the faces of $G$ include the right, bottom, and bottom-right neighbors of $f$. Again $G'$ is a simple grid graph, and the inductive hypothesis applied to $G'$ leads to the desired result for $G$.

In case $3$, the faces of $G$ include the right and bottom neighbors of $f$ but not the bottom-right neighbor.  Here $G'$ is not simple but is obtained from two disjoint simple grid graphs by identifying two vertices into one.  Applying the inductive hypothesis to both of them leads to the desired result for $G$.

\end{proof}

\begin{figure}[ht!]\centering
   \scalebox{.35}{\includegraphics{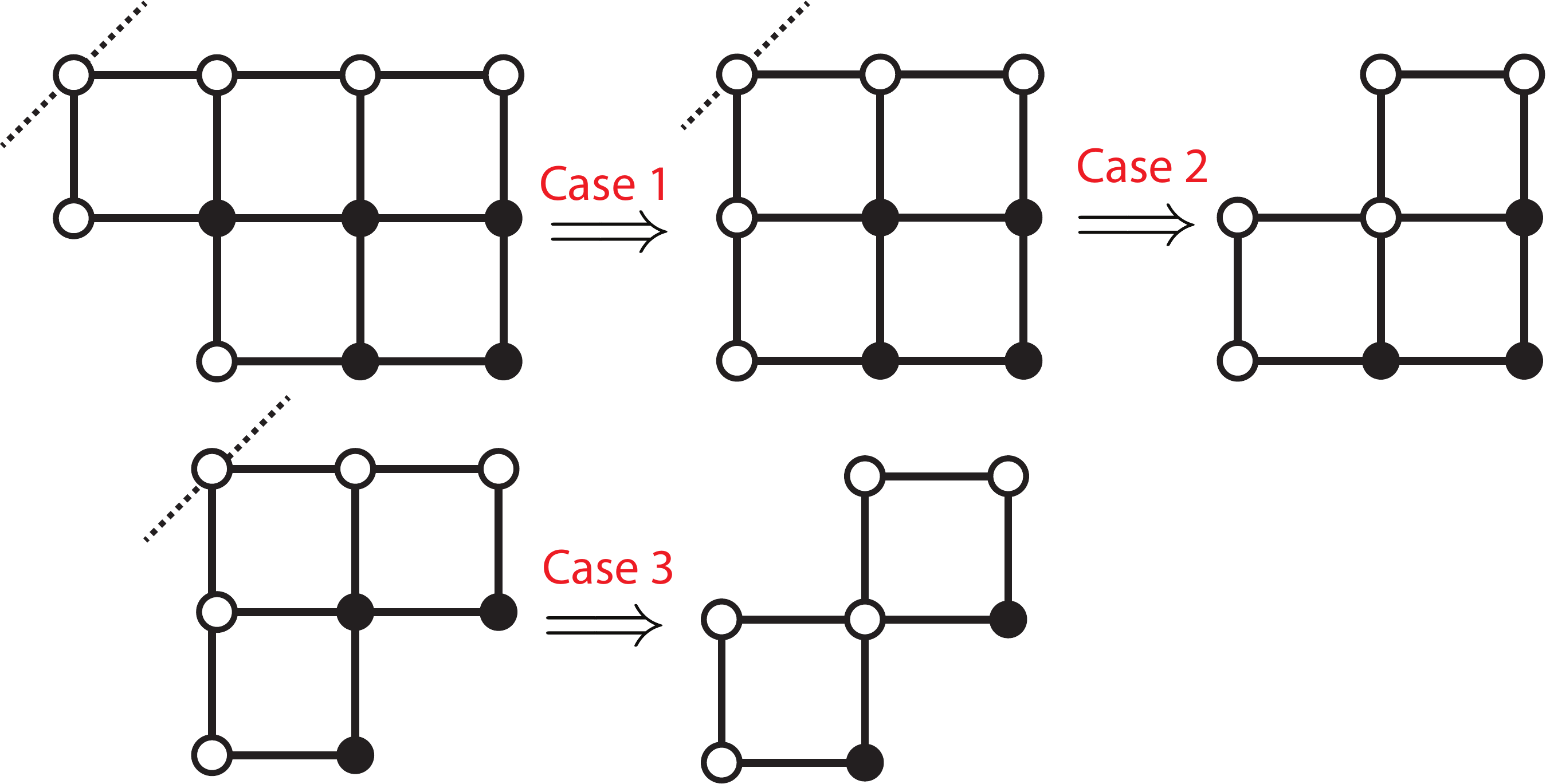}}
\caption{Inductive proof of Lemma~\ref{L:count}}\label{F:half_boundary}
   \end{figure}

The following proposition establishes the equivalence of the three expression for $m$ in Theorem~\ref{T:main}.
\begin{prop}\label{P:threeint} If $G$ is a simple grid graph, then
$$\text{Area}(G) = |V(G)|-\frac 12 |\partial G| - 1 = |V(G)|-|\hB G|.$$
\end{prop}
\begin{proof}
The second equality comes from Lemma~\ref{L:count}.  As mentioned in the introduction, the first equality follows from Pick's Theorem.  Alternatively, $\text{Area}(G) = |V(G)|-|\hB G|$ because the set of faces of $G$ correspond one-to-one with $V(G)-\hB G$ by matching each face with its bottom-right corner.
\end{proof}
\section{A lower bound on $\tau(G)$}
The goal of this section is to prove the lower bound in Theorem~\ref{T:main}.  In fact, we will prove the following generalization to (not necessarily simple) grid graphs:

\begin{theorem}\label{T:lower}If $G$ is a grid graph, then $$\tau(G)\geq\bm^m,$$
where $m = |V(G)|- |\hB G|$.
\end{theorem}
Theorem~\ref{T:lower} is an immediate consequence of the following:
\begin{prop}\label{P:lower_m_bound}
Let $G$ be a grid graph and $v\in V(G)$.  If $v\notin\hat{\partial} G$, then $m_v\geq \mathfrak{b}$.
\end{prop}

For the remainder of this section, we assume that $G$ is a grid graph, we fix a vertex $v\in V(G)$ and we assume that $v\notin\hat{\partial} G$, with the goal proving that $m_v\geq \mathfrak{b}$.

Denote the coordinates of $v$ as $v=(x,y)$ and denote the top and left neighbors of $v$ as $a=(x,y+1)$ and $b=(x-1,y)$.  Since $v\notin\hat{\partial} G$, we know that $\overline{va},\overline{vb}\in E(H_v)$.

\begin{lem}\label{L:P} Let $P_v$ denote the probability that a uniformly randomly selected member of $\mathcal{T}(H_v)$ contains both $\overline{va}$ and $\overline{vb}$.  Then $0<P_v<1$ and
$$m_v = \frac{2}{1-P_v}.$$
\end{lem}
\begin{proof}
Partition the members of $\mathcal{T}(H_v)$ into three sets, $\mathcal{T}(H_v)= T_1\cup T_2\cup T_3$, according to whether they:
\begin{enumerate}
\item[($T_1$)] Contain $\overline{va}$ but not $\overline{vb}$
\item[($T_2$)] Contain $\overline{vb}$ but not $\overline{va}$
\item[($T_3$)] Contain $\overline{va}$ and $\overline{vb}$
\end{enumerate}

Since ${}^\Box v$ is a subgraph of $H_v$, it is straightforward to see that all three sets in this partition are nonempty.  For example, $T_3$ is nonempty because a member of $T_3$ can be obtained from any member of $\mathcal{T}(H_v')$ by adding $\overline{va}$ and $\overline{vb}$ and removing any other edge of the resulting cycle that this creates.

Furthermore, $|\mathcal{T}(H'_v)|=|T_1|$ because ``adding the edge $\overline{va}$'' is a bijection between these sets.  Similarly $|\mathcal{T}(H'_v)|=|T_2|$.  Thus,
$m_v = \frac{\tau(H_v)}{\tau(H_v')} = \frac{2\tau(H'_v)+|T_3|}{\tau(H'_v)}$.  Solving $P_v = \frac{|T_3|}{2\tau(H'_v)+|T_3|}$ for $|T_3|$ and substituting completes the proof.
\end{proof}

\begin{lem}\label{L:E} Let $E_v$ denote the probability that a simple random walk on $H_v$ starting at $v$ ``escapes to $b$,'' which means that it reaches $b$ before returning to $v$.  We have:
$$m_v = \frac{2E_v}{2E_v-1}.$$
\end{lem}
\begin{proof}
Define $P_v$ as in Lemma~\ref{L:P}.  We can better understand $P_v$ via the Aldous-Broder algorithm for generating a uniformly random spanning tree of a connected graph~\cite{Aldous},\cite{Broder} (Wilson's algorithm from~\cite{Wilson} would also work here).  Their algorithm works as follows.  Start at any vertex and do a simple random walk.  Each time a vertex is first encountered, mark the edge from which it was encountered.  When all vertices have been encountered, the set of marked edges is a uniformly random spanning tree.

We apply the Aldous-Broder algorithm as follows.  Let $\mathcal{W}(a)$ denote a simple random walk starting at $a$ on the connected component of $H_v$ that contains $a$.  Denote this connected component as $H_v^0$, and note that it also contains $v$ and $b$ because $v\notin\hat{\partial} G$.

It is straightforward to see that $P_v$ equals the probability that, in the walk $\mathcal{W}(a)$, the vertex $b$ is first encountered along the edge $\overline{vb}$.  In fact, this is the only way in which the set of marked edges will end up containing both $\overline{va}$ and $\overline{vb}$.

Next let $\mathcal{W}(v)$ denote a simple random walk starting at $v$ on $H_v^0$.  Here is a review of the definitions of $P_v$ and $E_v$ together with a new definition of $Q_v$:
\begin{itemize}
\item $P_v = $ the probability in $\mathcal{W}(a)$ that $b$ is first encountered along $\overline{vb}$.
\item $Q_v = $ the probability in $\mathcal{W}(a)$ of reaching $v$ before reaching $b$.
\item $E_v = $ the probability in $\mathcal{W}(v)$ of reaching $b$ before returning to $v$.
\end{itemize}
Since $E_v$ is the probability of escaping to $b$ on the first step plus the probability of escaping after more than one step, we have:
\begin{equation}\label{E:QE}E_v = \frac 12 + \frac 12(1-Q_v)
\end{equation}

It remains to relate $P_v$ and $Q_v$.  For this, let $P(k)$ denote the probability in $\mathcal{W}(a)$ that $b$ is first encountered along $\overline{vb}$ immediately following the walk's $k^\text{th}$ visit to $v$.  We have:
\begin{equation}\label{E:PQ}
P_v = \sum_{k\geq 1} P(k) = \sum_{k\geq 1} \left(\frac{Q_v}{2}\right)^k = \frac{Q_v}{2-Q_v}
\end{equation}
Combining Equations~\ref{E:QE} and~\ref{E:PQ} with Lemma~\ref{L:P} yields the following expressions for the multiplier:
$$m_v = \frac{2}{1-P_v} = \frac{2-Q_v}{1-Q_v} = \frac{2E_v}{2E_v-1}.$$
\end{proof}

The problem is now reduced to understanding the escape probability $E_v$.  A standard trick in the literature is to bound escape probabilities using Rayleigh's Monotonicity Laws, whose intuition comes from the long studied connection between random walks and electrical circuits.  We recommend~\cite{DS} for an elementary introduction to this connection and to Rayleigh's Laws.  We'll require the following special case:
\begin{prop}[Rayleigh's Monotonicity Law]\label{P:Rayleigh} Let $\tilde{H}$ be a connected graph, let $H$ be a subgraph of $\tilde{H}$, and let $v_0,b_0\in V(H)$ be distinct vertices.  Assume that $H$ contains all edges in $\tilde{H}$ incident to $v_0$.  Let $\tilde{E}$ (respectively $E$) denote the probability that a simple random walk on $\tilde{H}$ (respectively on $H$) starting at $v_0$ ``escapes to $b_0$,'' which means it reaches $b_0$ before returning to $v_0$.  Then $E\leq \tilde{E}$.
\end{prop}

Thus, there is a greater probability of escape on the larger graph than on the smaller subgraph.  In our application of Rayleigh's Law, the smaller graph will be $H_v$, while the larger will be the infinite subgraph, $\U$, of $\mathcal{L}(\Z^2)$ whose vertex set is:
\begin{equation}\label{E:U_def}V(\U) = \{(x,y)\in\Z^2\mid y\geq 1\text{ or }(y=0\text{ and } x\leq 0)\}.
\end{equation}
We can think of $V(\U)$ as the set of points of $\Z^2$ prior to (and including) the origin $\mathbf{0}=(0,0)$ in the words-on-a-page sense.  After applying a translation for notational convenience, we can assume that $v$ is positioned at the origin; that is, we can assume that $v=\mathbf{0}=(0,0)$, $a=(0,1)$, and $b=(-1,0)$.  With this understanding, $H_v$ is a subgraph of $\U$.

\begin{lem}\label{L:N2} Let $E(\infty)$ denote the probability that a simple random walk on $\U$ starting at $v=(0,0)$ escapes to $b=(-1,0)$.  Then $E(\infty) = \frac{\bm}{2(\bm-1)}$.
\end{lem}

We postpone the proof of Lemma~\ref{L:N2} until the end of the next section.  For now, we will use the lemma to finish off Proposition~\ref{P:lower_m_bound} and hence also Theorem~\ref{T:lower}.

\begin{proof}[Proof of Proposition~\ref{P:lower_m_bound}]
Rayleigh's Law gives $E_v\leq E(\infty)= \frac{\bm}{2(\bm-1)}$.
Note that $E_v>1/2$ because there is a probability $1/2$ of escaping to $b$ in the first step.  On the domain $E_v>1/2$, the function $m_v = \frac{2E_v}{2E_v-1}$ is decreasing.  Therefore $m_v\geq \frac{2E(\infty)}{2E(\infty) - 1}=\bm$.
\end{proof}

\begin{proof}[Proof of Theorem~\ref{T:lower}]
$$\tau(G) = \prod_{v\in V(G)} m_v \geq \prod_{v\in V(G)-\hB G} m_v
\geq \bm^m,$$
where $m = |V(G)|-|\hB G|$.
\end{proof}

\section{An upper bound on $\tau(G)$}
The goal of this section is to prove the upper bound of Theorem~\ref{T:main}. Figure~\ref{F:counterexample} shows that this upper bound is false in the non-simple case (with $m$ re-expressed in terms of $\hB G$ as in Theorem~\ref{T:lower}).
This upper bound will follow immediately from Lemma~\ref{L:toplefty} together with the following:

\begin{prop}\label{C:1} If $G$ is a simple grid graph and $v\in V(G)$ with $v\notin\hB G$, then $m_v\leq 4$
\end{prop}

We will prove this proposition (and more general upper bounds on $m_v$) via Rayleigh's Law by comparing $H_v$ to a smaller subgraph constructed as follows.

For each integer $k\geq 1$, define $\tilde{\U}(k)$ to be the subgraph of $\U$ (from Equation~\ref{E:U_def}) induced by all vertices within distance $k$ from $\mathbf{0}$, and then obtain $\U(k)$ from $\tilde{\U}(k)$ by removing all vertices of degree $1$ and their adjacent edges.  That is,
\begin{gather*}
V(\tilde{\U}(k))  = \{p\in V(\mathcal{U})\mid\text{dist}(\mathbf{0},p)\leq k\},\\
V(\U(k))  = \{p\in V(\tilde{\U}(k))\mid \text{degree}(p)\neq 1\},
\end{gather*}
where ``dist'' is the edge distance of the graph. The first few are shown in Figure~\ref{F:distspheres}.

\begin{figure}[ht!]\centering
   \scalebox{.30}{\includegraphics{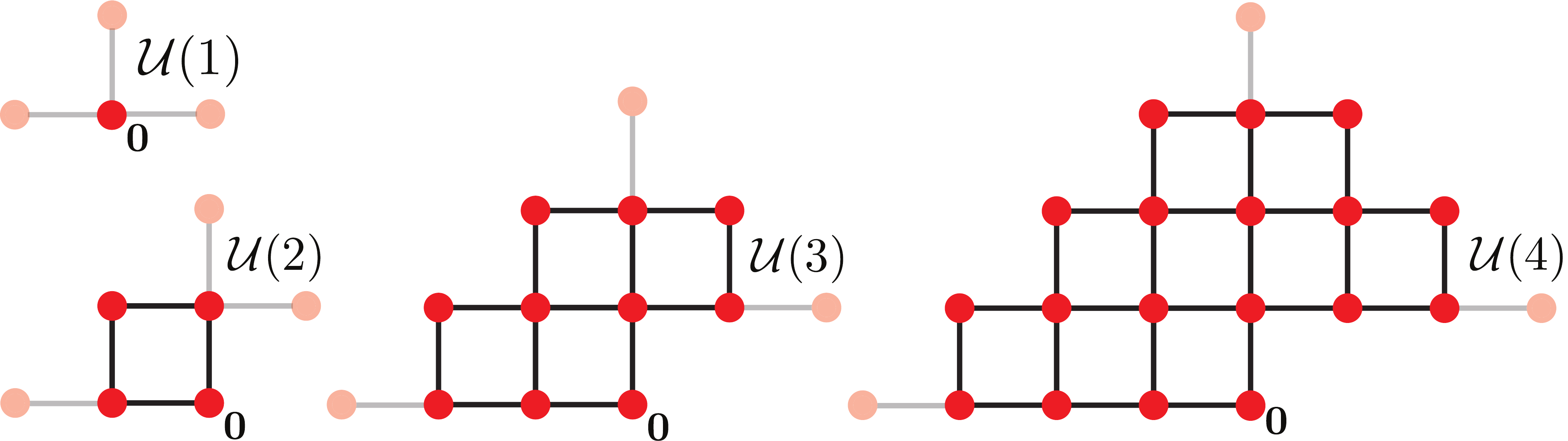}}
\caption{$\U(k)$ for $k\in\{1,2,3,4\}$.  The transparent vertices and edges belong to $\tilde{\U}(k)$ but not $\U(k)$.}\label{F:distspheres}
   \end{figure}

Assume for the remainder of the section that $G$ is a grid graph and $v\in V(G)$.  As in the previous section, assume (after applying a translation) that $v=\mathbf{0}=(0,0)$ so that $H_v\subset\U$. Define:
$$d_v = \max\{k\mid \U(k)\subset H_v\}.$$
Notice that $v\in\hB G$ if and only if $d_v=1$.

\begin{lem}\label{L:F_def} If $v\notin\hB G$ (or equivalently if $d_v\geq 2$), then $m_v\leq F(d_v)$, where $F$ is a function explicitly defined in the proof below, whose first few values are given in Table~\ref{T:F}.
\end{lem}

\begin{table}[ht]
\caption{Some values of $F$ rounded to 4 decimals}
\centering
\begin{tabular}{c c}
\hline\hline
$k$ & $F(k)$  \\ [0.5ex]
\hline
2 & 4  \\
3 & 3.4833  \\
4 & 3.3486  \\
5 & 3.2936  \\
\vdots & \vdots  \\ [1ex]
12 & 3.2193 \\
 & $\downarrow$ \\
  & $\mathfrak{b}\approx 3.2099$\\
\hline
\end{tabular}
\label{T:F}
\end{table}
\begin{proof}
Set $k=d_v$.  Define $Q_v$ and $E_v$ as in the proof of Lemma~\ref{L:E}, in which are found the relations:
$$m_v = \frac{2E_v}{2E_v-1}= \frac{2-Q_v}{1-Q_v}.$$

Analogously define $Q(k)$ and $E(k)$ with respect to random walks in $\U(k)$; that is:
\begin{itemize}
\item $Q(k)$ is the probability that a simple random walk in $\U(k)$ starting at $a$ reaches $v$ before reaching $b$.
\item $E(k)$ is the probability that a simple random walk in $\U(k)$ starting at $v$ escapes to $b$.
\end{itemize}
Define
$$F(k) = \frac{2E(k)}{2E(k)-1}=\frac{2-Q(k)}{1-Q(k)}.$$

Rayleigh's Monotonicity Law implies that $E_v\geq E(k)$ and therefore that $m_v\leq F(k)$.

To explicitly compute $F(k)$, it will suffice to compute $Q(k)$ via the method of~\cite[Section 1.2.6]{DS}, which we briefly review here.  Regard the random walk on $\U(k)$ starting at $a$ as an absorbing Markov chain with absorbing states $\{v,b\}$.  Index the vertices of $\U(k)$ with these absorbing states listed first, so the transition matrix of the Markov chain has the block form
$\left( \begin{matrix}\mathbf{I} & \mathbf{0}\\ \mathbf{R} & \mathbf{Q} \end{matrix}\right).$ The absorbtion probabilities are given by $\mathbf{B} = (\mathbf{I}-\mathbf{Q})^{-1}\mathbf{R}$.  Thus $Q(k)$ equals the entry of $\mathbf{B}$ whose column corresponds to the absorbing state $v$ and whose row corresponds to the non-absorbing state $a$.
\end{proof}

The function $F:\{2,3,...\}\ra\mathbb{R}$ defined in the previous proof has the following properties.
\begin{lem}\label{L:Fnonicr}
$F$ is non-increasing, and $\lim_{k\ra\infty}F(k) = \mathfrak{b}$.
\end{lem}
\begin{proof}
The assertion that $F$ is non-increasing follows immediately from Rayleigh's Monotonicity Law because $\U(k)\subset\U(k+1)$.

The limit claim will come from Equation~\ref{E:Bulklim}.  If $S(1)\subset S(2)\subset\cdots$ is a sequence of concentric squares in $\mathcal{L}(\Z^2)$ centered at $\mathbf{0}$ with diameters going to infinity, then Equation~\ref{E:Bulklim} implies
\begin{equation}\label{E:bulklim}
\lim_{n\ra\infty}\frac{\ln(\tau(S(n)))}{|V(S(n))|} = \ln\bm.
\end{equation}
We will argue that any value for $\lim_{k\ra\infty}F(k)$ different from $\mathfrak{b}$ would contradict Equation~\ref{E:bulklim}.

For this, first suppose that $\lim_{k\ra\infty}F(k)$ were strictly \emph{less} than $\bm$, so there would exist $k_0\in\N$ and $\epsilon>0$ such that $F(k_0)<\bm-\epsilon$.  But then for every $n$, every vertex $v_0\in V(S(n))$ at distance more than $k_0$ from the relevant edges of $S(n)$ (left, top, right) would have multiplier value $m_{v_0}<\bm-\epsilon$.  The fraction of vertices $v_0\in V(S(n))$ to which this applies approaches $100\%$ as $n$ grows.  The vertices to which this doesn't apply can't compensate because their multiplier values are all bounded above by $4$.  This contradicts Equation~\ref{E:bulklim}.

Next suppose that $\lim_{k\ra\infty}F(k)$ were strictly \emph{greater} than $\bm$, so there exists $\epsilon>0$ such that $F(k)>\bm+\epsilon$ for all $k\geq 2$.  For every fixed $n$, it is possible to choose $k$ sufficiently large so that for every $v_0\in S(n)$, $\U(k)$ contains $H_{v_0}$ and hence (assuming $v_0\notin\hat{\partial} S(n)$) we have $m_{v_0}\geq F(k)>\bm+\epsilon$ (by the arguments of Section 5, with $\U(k)$ replacing $\U$).  Thus, for every fixed $n$, every vertex $v_0\in S(n)$ has multiplier $m_{v_0}$ equal to $1$ if $v_0\in\hat{\partial} S(n)$, or greater than $\bm+\epsilon$ if $v_0\notin\hat{\partial} S(n)$.  Since the fraction of vertices in $\hat{\partial} S(n)$ goes to zero and $n$ goes to infinity, this contradicts Equation~\ref{E:bulklim}.
\end{proof}

The value $F(2)=4$ in Table~\ref{T:F} is exact (not rounded). Proposition~\ref{C:1} is an immediate consequence of this value.

We now use Lemma~\ref{L:Fnonicr} to fill in a missing proof from Section 5.

\begin{proof}[Proof of Lemma~\ref{L:N2}]
Let $E(k)$ denote the probability that a simple random walk starting at $v=(0,0)$ on $\mathcal{U}(k)$ escapes to $b=(-1,0)$.  Then
$$E(\infty) = \lim_{k\ra\infty} E(k) = \lim_{k\ra\infty} \frac{F(k)}{2(F(k)-1)}= \frac{\bm}{2(\bm-1)}.$$
\end{proof}

Finally, we prove the upper bound of Theorem~\ref{T:main} as a quick consequence of Lemma~\ref{L:toplefty} and Proposition~\ref{C:1}.
\begin{proof}[Proof of upper bound of Theorem~\ref{T:main}]
$$\tau(G) = \prod_{v\in V(G)} m_v = \prod_{v\in V(G)-\hB G} m_v
\leq 4^m,$$
where $m = |V(G)|-|\hB G| = |V(G)|-\frac 12 |\partial G|-1$ (by Lemma~\ref{L:count}).

\end{proof}

An improvement on the upper bound of Theorem~\ref{T:main} can be obtained by considering the level sets of $d$.  That is, for each $k\geq 1$ define:
$$G^k = \{v\in V(G)\mid d_v=k\}.$$
Notice that $G^1 = \hB G$.  Assuming that $G$ is simple, Lemma~\ref{L:F_def} gives:
\begin{equation}\label{E:last}\ln(\tau(G))\leq\sum_{k\geq 2} \ln(F(k))\cdot|G^k|.\end{equation}

Equation~\ref{E:last} is stronger than the upper bound of Theorem~\ref{T:main}, but is it is not clear whether Equation~\ref{E:last} is necessarily stronger for all simple grid graphs than Theorem~\ref{T:Lyons}.
\section{Application to redistricting}
In this section, we apply our results and techniques to shed light on the redistricting question mentioned in the introduction: why does a map's cut-edge count have such a strong negative correlation with the \emph{spanning tree score}, which means the log of the product of the numbers of spanning trees of its districts?  Independent work by Procaccia and Tucker-Folz relate these two measurements for general planar graphs~\cite{JamieAriel}; in the case of grid graphs, our results are complimentary to theirs.

A common starting point of modern redistricting models is a graph $G$ whose vertices represent the precincts of a state.  Two vertices are connected by an edge if the corresponding precincts share a geographic boundary with non-zero length.  A \emph{map} is a partition of $G$ into subgraphs called \emph{districts}, which are required to satisfy certain legal requirements.

To shed light on the general situation, we will study the special case of simple grid graphs.  More precisely, let $G$ denote a simple grid graph, let $\{V_1,...,V_K\}$ denote a partition of $V(G)$, and let $\{G_1,...,G_K\}$ denote subgraphs of $G$ induced by these vertex sets, which we call \emph{districts}.  We assume that each $G_i$ is a simple grid graph.

Let $\C$ denote the set of \emph{cut edges}, which means the edges between pairs of vertices of $G$ that belong to different districts.  The value $|\C|$ is frequently used as a discrete measurement of the map's overall compactness; see~\cite{DT} for advantages of this measurement compared to other compactness measurements.  Figure~\ref{F:scatter} exhibits a very strong negative correlation between $|\C|$ and the spanning tree score $\ln\left( \prod \tau(G_i) \right)$ for a ensemble of 1000 partitions of the square with $30^2$ vertices into $9$ districts.  This ensemble was created with the ReCom algorithm~\cite{Recom} with $5\%$ population deviation using the pictured tic-tac-toe arrangement as the initial partition.   Our goal is to account for this negative correlation.

\begin{figure}[ht!]\centering
   \scalebox{.85}{\includegraphics{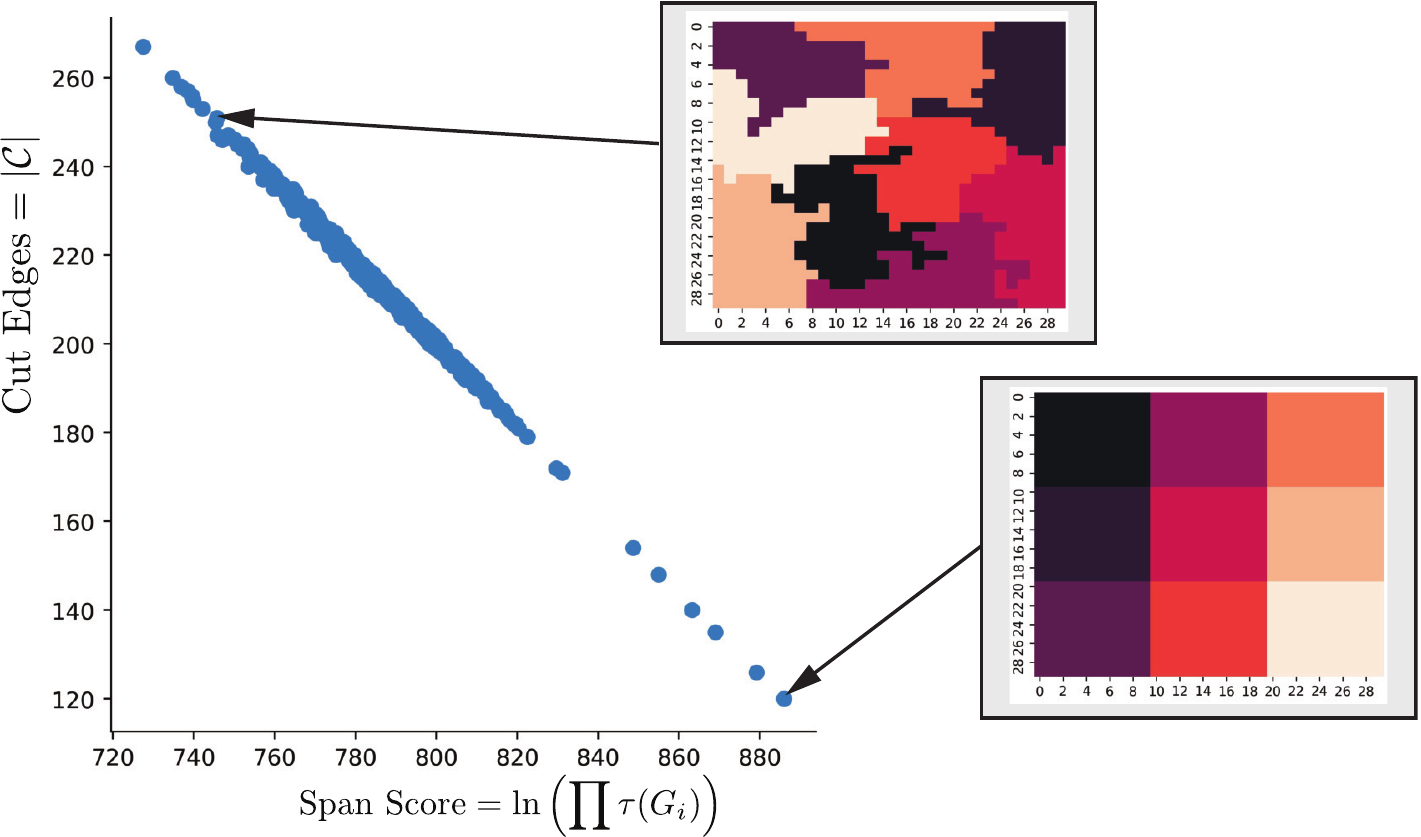}}
\caption{Spanning tree score vs. cut edge count for an ensemble of 1000 partitions of the square with $30^2$ vertices into $9$ districts}\label{F:scatter}
   \end{figure}

The main result of this section is:
\begin{prop}\label{P:redistrict}
$$|\C| =\underbrace{\text{Area}(G) +K-1}_{\text{denote as }C_1} - \sum_{i=1}^K \text{Area}(G_i).$$
\end{prop}

Notice that $C_1$ is a constant that does not depend on the partition.   We will now apply Proposition~\ref{P:redistrict} by observing that Theorem~\ref{T:main} implies that for each $i$, there exists $b_i\in(\bm,4]$ such that $\tau(G_i) = (b_i)^{\text{Area}(G_i)}$.
If we substitute $\text{Area}(G_i)=\frac{\ln(\tau(G_i))}{\ln(b_i)}$ and use that $b_i\in(\bm,4]$, then we obtain the following immediately from Proposition~\ref{P:redistrict}:
\begin{equation}\label{E:boundss}
C_1 - \frac{1}{\ln(\mathfrak{b})} \cdot \ln\left( \prod \tau(G_i) \right)
\leq  |\C| \leq C_1 - \frac{1}{\ln(4)} \cdot \ln\left( \prod \tau(G_i) \right).
\end{equation}

Figure~\ref{F:bounds} contains the same data as Figure~\ref{F:scatter} (zoomed out in order to show the axes) with the upper and lower bounds of Equation~\ref{E:boundss} displayed as red lines.  The slopes of these red lines are $- \frac{1}{\ln(\mathfrak{b})}$ and $-\frac{1}{\ln(4)}$. Their common vertical intercept is $C_1=29^2 + 9 -1 = 849$.  Notice that all of the data points lie between the two red lines, even though the maps in this ensemble do not satisfy all of our hypotheses -- their districts are not all simple.

\begin{figure}[ht!]\centering
   \scalebox{.65}{\includegraphics{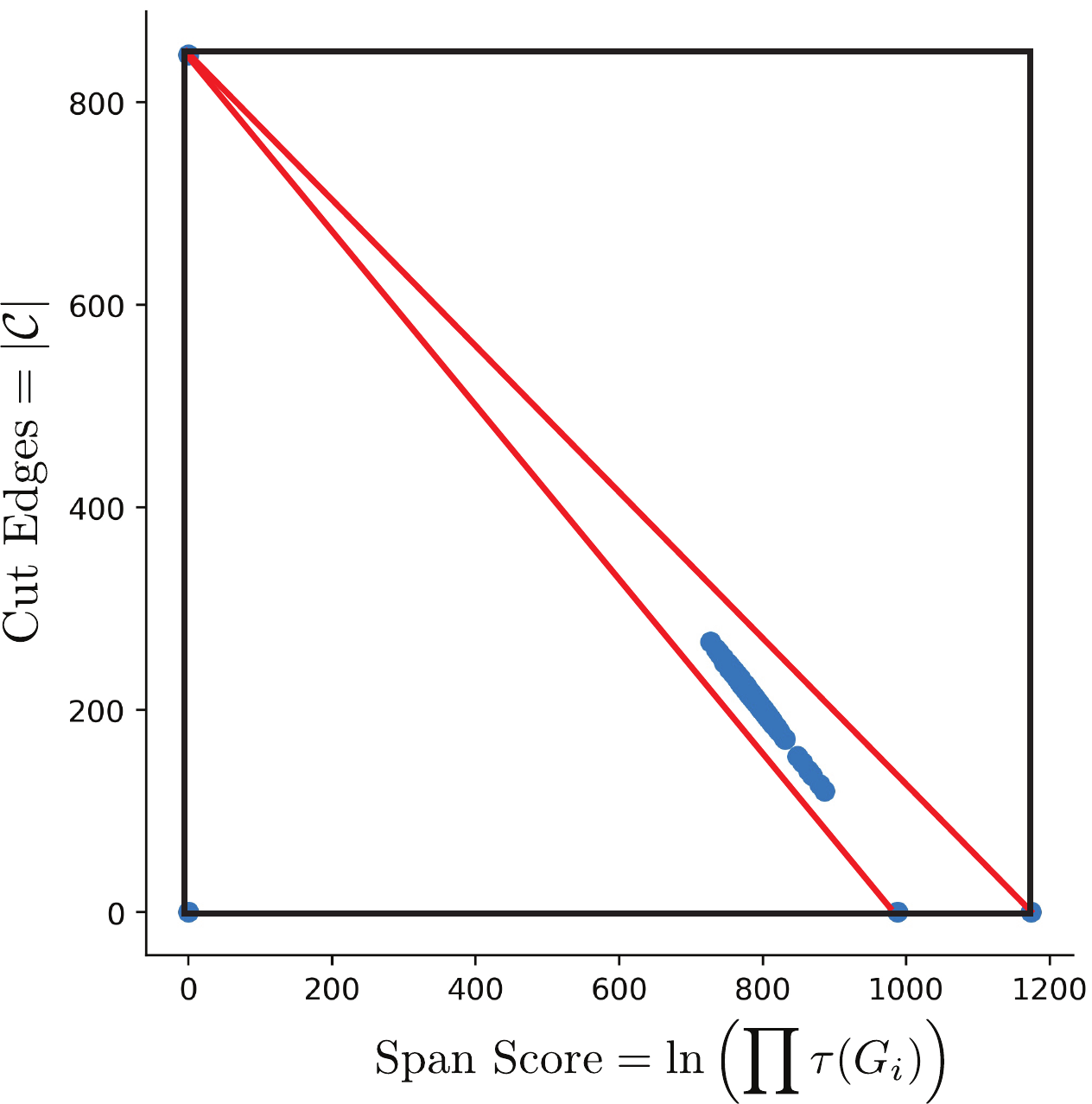}}
\caption{The red lines represent the inequalities of Equation~\ref{E:boundss}}\label{F:bounds}
   \end{figure}

\begin{proof}[Proof of Proposition~\ref{P:redistrict}]
Let $F$ denote the set of faces of $G$ that are not faces of any of the districts.  It will suffice to prove that $|\C|=|F|+K-1$.

It is possible to select a subset $S\subset\C$ of size $K-1$ such that $S$ connects the districts into a spanning tree.  More precisely, $S$ induces a spanning tree, $T_S$, on the \emph{district quotient graph} of $G$, which is defined to contain one vertex for each district, and to have an edge between each pair of vertices if the corresponding pair of districts is connected by at least one cut edge.  Figure~\ref{F:lazyriver} provides an illustration in which the districts are dark grey, the faces of $F$ are light grey, the members of $S$ are dashed red lines, and the members of $\C-S$ are solid blue lines.

It will suffice to find a bijection from $F$ to $\C-S$.  For this, we will consider $F$ as a graph in which a pair $f_1,f_2\in F$ are connected by an edge if they are adjacent across a member of $\C-S$.  Considered in this way, $F$ is acyclic because $T_S$ is connected.  Thus, $F$ is a union of disjoint trees.  We'll call $f\in F$ an \emph{end face} if it is adjacent across an edge of $\C-S$ with a face of $\mathcal{L}(\Z^2)$ that's not a face of $G$.  Since $T_S$ is acyclic, each connected component of $F$ contains at least one end face.

Imagine following a path in $F$ and marking the faces and cut edges crossed along the way.  Since faces and cut edges alternate, we can insure we mark an equal number of each by starting with a face and ending with a cut edge.  Let's call such a path a \emph{good path}.  To build a bijection of $F$ with $\C-S$, it will suffice to find a finite collection of good paths that together mark all of the faces in $F$ and all of the cut edges in $\C-S$.  This can be achieved by repeating the following two steps until all faces have been marked:
\begin{enumerate}
\item Select any face of $f\in F$ that hasn't yet been marked.
\item There exists a path in $F$ from $f$ to an end face.  Traverse this path (marking the faces and edges along the way) until either reaching this end face or reaching a previously marked face.
\end{enumerate}
In Figure~\ref{F:lazyriver}, one possible collection of good paths is illustrated in green.  When this algorithm terminates, all edges of $\C-S$ must be marked because any unmarked edge could be added to $S$ without creating a cycle in $T_S$.

\end{proof}

\begin{figure}[ht!]\centering
   \scalebox{2}{\includegraphics{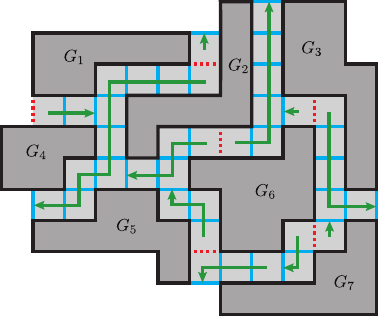}}
\caption{The green paths provide a bijection between $F$ (the light grey faces) and $\C-S$ (the blue cut edges).}\label{F:lazyriver}
   \end{figure}

We end by warning that $\text{Area}(G_i)$ does not equal the area of a geographic district because in our model $G_i$ denotes the \emph{dual graph} of a district.  Suppose we back our model up a step by first letting $G_i^*$ be a grid graph whose faces comprise the region enclosed in the $i^{\text{th}}$ geographic district, and then defined $G_i$ as its dual, so that the vertices of $G_i$ are the midpoints of the faces of $G_i^*$.  In this case, the boundary of each geographic district $G_i^*$ would run down the center of $F$.

\bibliographystyle{amsplain}

\end{document}